\documentclass[11pt,a4paper,twoside,reqno]{amsart}
\pdfoutput=1
\title[Balanced metrics on six-dimensional manifolds of cohomogeneity one]{On the existence of balanced metrics on six-manifolds of cohomogeneity one}
\author{Izar Alonso}
\address[I. Alonso]{Department of Mathematics \\ Rutgers University \\Piscataway, NJ 08854\\ USA}
\email{izar.alonso@rutgers.edu}
\author{Francesca Salvatore}
\address[F. Salvatore]{Dipartimento di Matematica ``G. Peano'' \\ Universit\`a degli Studi di Torino\\ Via Carlo Alberto 10\\10123 Turin\\ Italy}
\email{francesca.salvato92@gmail.com}
\subjclass[2010]{53C25}
\keywords{Cohomogeneity one actions, Balanced metrics, $\text{SU}(3)$-structures}
\thanks{The first author was supported by the EPSRC.
The second author was supported by GNSAGA of INdAM}

\usepackage{hyperref,url}
\usepackage{amssymb,latexsym}
\usepackage{amsmath,amsthm}
\usepackage[latin1]{inputenc}
\usepackage{enumerate}
\usepackage{xcolor}
\usepackage{mathtools}
\usepackage[italian]{varioref}
\makeatletter
\labelformat{equation}{\tagform@{#1}}
\makeatother
\usepackage{multirow}
\usepackage{graphicx}
\usepackage{enumitem}
\usepackage{float}
\newcommand\xrowht[2][0]{\addstackgap[.5\dimexpr#2\relax]{\vphantom{#1}}}
\theoremstyle{plain}
\newtheorem{theorem}{Theorem}[section]
\newtheorem*{theorem*}{Theorem}
\newtheorem*{theoremA}{Theorem A}
\newtheorem*{theoremB}{Theorem B}

\newtheorem{corollary}[theorem]{Corollary}
\newtheorem{proposition}[theorem]{Proposition}

\theoremstyle{definition}
\newtheorem{definition}[theorem]{Definition}

\theoremstyle{remark}
\newtheorem{remark}[theorem]{Remark}

\numberwithin{equation}{section} 

\newcommand\norm[1]{\left\lVert#1\right\rVert}

\usepackage{scalerel,stackengine}
\stackMath
\newcommand\widehatt[1]{%
\savestack{\tmpbox}{\stretchto{%
  \scaleto{%
    \scalerel*[\widthof{\ensuremath{#1}}]{\kern-.6pt\bigwedge\kern-.6pt}%
    {\rule[-\textheight/2]{1ex}{\textheight}}
  }{\textheight}%
}{0.5ex}}%
\stackon[1pt]{#1}{\tmpbox}%
}

\begin{document}

\begin{abstract}
We consider balanced metrics on complex manifolds with holomorphically trivial canonical bundle, most commonly known as balanced \text{SU}(n)-structures. 
Such structures are of interest for both Hermitian geometry and string theory, since they provide the ideal setting for the Hull--Strominger system. 
In this paper, we provide a non-existence result for balanced non-K\"ahler $\text{SU}(3)$-structures which are invariant under a cohomogeneity one action on  simply connected six-manifolds.
\end{abstract}
\maketitle
\section{Introduction}
A $\text{U}(n)$-structure on a $2n$-dimensional smooth manifold $M$ is the data of a Riemannian metric $g$ and a $g$-orthogonal almost complex structure $J$. The pair $\left(g,J\right)$ is also known as an almost Hermitian structure on $M$.
When $J$ is integrable, i.e., $\left(M,J\right)$ is a complex manifold, the pair $\left(g,J\right)$ defines a Hermitian structure on $M$.
In this case, the metric $g$ is  called \emph{balanced} when ${\text{d}}\omega^{n-1}=0$, $\omega\coloneqq g\left(J\cdot,\cdot\right)$ denoting the associated fundamental form, and we shall refer to $\left(g,J\right)$ as a balanced $\text{U}(n)$-structure on $M$.
Balanced metrics have been extensively studied in \cite{BedVezz, FinoGranVezz, FinoVezz1, FinoVezz2, FuLiYau, Mic82, PhongPicardZhang} (see also the references therein). 

Balanced metrics are also interesting in the context of $\text{SU}(n)$-structures, especially in the six-dimensional case, thanks to their applications in physics.
An $\text{SU}(n)$-structure $\left(g,J,\Psi\right)$ on a $2n$-dimensional smooth manifold $M$, is a $\text{U}(n)$-structure $\left(g,J\right)$ on $M$ together with a $\left(n,0\right)$-form of nonzero constant norm $\Psi=\psi_+ + i\psi_-$ satisfying the normalization condition $\Psi \wedge \overline{\Psi}=(-1)^{\frac{n(n+1)}{2}}(2i)^n \frac{\omega^n}{n!}$.
An $\text{SU}(n)$-structure $\left(g,J,\Psi\right)$ on $M$ with underlying balanced $\text{U}(n)$-structure $\left(g,J\right)$ for which ${\text{d}}\omega\neq 0$ and ${\text{d}}\Psi=0$ will be referred to as a balanced non-K\"ahler $\text{SU}(n)$-structure.

In 1986, Hull and Strominger \cite{Hull, Strominger}, 
independently, introduced a system of \textsc{pde}s, now known as the \emph{Hull--Strominger system}, to formalize certain properties of the inner space model used in string theory.
Let $M$ be a $2n$-dimensional complex manifold equipped with a nowhere-vanishing holomorphic $\left(n,0\right)$-form $\Psi$ and let $E$ be a holomorphic vector bundle on $M$ endowed with the Chern connection. 
The  Hull--Strominger system consists of a set of \textsc{pde}s involving a pair of Hermitian metrics $\left(g,h\right)$  on $\left(M,E\right)$.
One of these equations dictates the metric $g$ on $M$ to be conformally balanced, more precisely ${\text{d}}\left(\norm{\Psi}_{\omega}\omega^{n-1}\right)=0$, where $\norm{\Psi}_{\omega}$ is the norm of $\Psi$ given explicitly by $\Psi \wedge \overline{\Psi}=(-1)^{\frac{n(n+1)}{2}}\frac{i^n}{n!}  \norm{\Psi}_{\omega}^2  \omega^n$. When one assumes all structures to be invariant under the smooth action of a certain Lie group $G$, the aforementioned condition reduces to the balanced equation ${\text{d}}\omega^{n-1}=0$, since the norm of $\Psi$ is constant.
Notice that in these cases $(g,J,\Psi)$ is a balanced $\text{SU}(n)$-structure on $M$, up to a suitable uniform scaling of $\Psi$. 

The issue of the existence and uniqueness of a general solution to the Hull--Strominger system is still an open problem. Nonetheless,  solutions have been found under more restrictive hypotheses; for the non-K\"ahlerian case, we refer the reader, for instance, to \cite{Fei, FeiHuangPicard, FuTsengYau, FuYau, Mario, Gran11, OtalUgaVilla}.
Other interesting solutions are given in \cite{FeiYau}, where a class of invariant solutions to the Hull--Strominger system on complex Lie groups was provided; these solutions extend to solutions on all compact complex parallelizable manifolds, by Wang's classification theorem \cite{Wang}. 
Moreover, in \cite{FinoGranVezz}, it was shown that a compact complex homogeneous space with invariant complex volume admitting a balanced metric is necessarily a complex parallelizable manifold. 
Then, the invariant solutions given in \cite{FeiYau} exhaust the complex compact homogeneous case. If one allows the Lie group acting on the homogeneous space to be real, many other solutions to the Hull--Strominger system are known in the literature, see for instance \cite{ Gran11, Pujia, PujiaUgarte, UgaVilla}. Then, one may wonder what happens in the cohomogeneity one case. 
A cohomogeneity one manifold $M$ is a connected smooth manifold with an action of a compact Lie group $G$ having an orbit of codimension one. 
Currently, there are no known examples of balanced non-K\"ahler $\text{SU}(n)$-structures invariant under a cohomogeneity one action. In this paper, we investigate their existence. 
In particular, we focus on the 
simply connected $2n=6$-dimensional case.
Recall that, when a cohomogeneity one manifold $M$ has finite fundamental group, then $M/G$ is homeomorphic to an interval $I$, see \cite{BerBer}. If we denote by $\pi:M\rightarrow M/G$ the canonical projection onto the orbit space,
we shall call $\pi^{-1}(t)$,  for every $ t\in \overset{\circ}{I}$, principal
orbits and the inverse images of the boundary points singular orbits. Denoting by $M^{\text{princ}}$ the union of all principal orbits, which is a dense open subset of $M$, and by $K$ the isotropy group of a principal point, which is unique up to conjugation along $M^{\text{princ}}$, the pair $\left(G,K\right)$ completely determines the principal part $M^{\text{princ}}$ of the cohomogeneity one manifold, up to $G$-equivariant diffeomorphisms.
Given a Lie group $H$, we denote its Lie algebra $\text{Lie}(H)$ by the corresponding gothic letter $\mathfrak{h}$. 

We first give a local result for the existence of balanced non-K\"ahler $\text{SU}(3)$-structures by working on $M^{\text{princ}}$.

\begin{theoremA}
Let $M$ be a six-dimensional simply connected cohomogeneity one manifold under the almost effective action of a connected Lie group $G$, and let $K$ be the principal isotropy group. Then, the principal part $M^{\text{\normalfont princ}}$ admits a $G$-invariant balanced non-K\"ahler $\text{\normalfont SU}(3)$-structure $\left(g,J,\Psi\right)$ if and only if 
$\left( \mathfrak{g}, \mathfrak{k} \right)=\left(\mathfrak{su}(2)\oplus\mathfrak{su}(2), \Delta \mathbb{R} \right)$ 
or
$M$ is compact and $\left( \mathfrak{g}, \mathfrak{k} \right)=\left(\mathfrak{su}(2)\oplus 2\mathbb{R},\{0\} \right)$.
\end{theoremA}
We then prove that, under the assumption that $(\mathfrak{g}, \mathfrak{k}) \neq (\mathfrak{su}(2) \oplus \mathfrak{su}(2), \Delta \mathbb{R})$, the local solutions cannot be extended to a global one.
This leads us to state our main theorem:
\begin{theoremB}
Let $M$ be a six-dimensional simply connected cohomogeneity one manifold under the almost effective action of a connected Lie group $G$, and let $K$ be the principal isotropy group. Assume $(\mathfrak{g}, \mathfrak{k}) \neq (\mathfrak{su}(2) \oplus \mathfrak{su}(2), \Delta \mathbb{R})$.
Then $M$ admits no $G$-invariant balanced non-K\"ahler $\text{\normalfont SU}(3)$-structures.
\end{theoremB}

In \cite{FuLiYau}, balanced metrics were constructed on the connected sum of $k \geq 2$ copies of $S^3 \times S^3$. However, it is not known whether $S^3 \times S^3$ admits  balanced  structures. 
In \cite[Example 1.8]{Mic82}, Michelsohn proved that $S^3\times S^3$ endowed with the Calabi--Eckmann complex structure does not admit any compatible balanced metric.
By \cite[Remark 1]{Alexandrov00}, in a manifold with six real dimensions,  there is no non-K\"ahler Hermitian metric which is simultaneously balanced and strong K\"ahler-with-torsion (a.k.a SKT). 
In \cite{FinoVezz1}, Fino and Vezzoni conjectured that on non-K\"ahler compact complex manifolds it is never possible to find
an SKT metric and also a balanced metric.
In \cite{Grantcharov08}, an example of a SKT structure on $S^3 \times S^3$ is provided.
The key case that needs to be tackled in Theorem B is precisely $S^3 \times S^3$.

The paper is organized as follows. 
In Sect.\ $2$, we review some basic facts about cohomogeneity one manifolds and $\text{SU}(3)$-structures which will be useful for our discussion. In Sect.\ $3$, we present our problem,
write a classification of the pairs $(\mathfrak{g},\mathfrak{k})$ that can occur, and use the hypothesis of simply connectedness to reduce the list to only three possibilities.
At the end of Sect.\ $3$, we state Theorem A, which we prove in Sect.\ $4$ via a case-by-case analysis.
Finally, in Sect.\ $5$, we prove Theorem B.

\vspace{12pt}

\noindent
{\bf Acknowledgements.}
The first named author wants to thank Andrew Dancer and Jason Lotay for introducing the problem and their support and help with it. The first named author would also like to thank Johannes Nordstr\"om for bringing up his concerns about equation (4.5) in the previous version of this document.
The second named author would like to thank Lucio Bedulli for introducing the problem and for useful conversations and comments and Anna Fino for her constant support, encouragement and patient guidance. The second named author would like to thank also Alberto Raffero and Fabio Podest\`{a}  for helpful discussions and remarks.

\section{Preliminary notions}
\subsection{Cohomogeneity one manifolds}

Here, we recall the basic structure of cohomogeneity one manifolds. For further details, see for instance \cite{AleBettiol, BerBer, Hoelscher, Hoelscher2, Ziller}.
\begin{definition}
A cohomogeneity one manifold is a connected differentiable manifold $M$ with an action $\alpha\colon G\times M\to M$ of a compact Lie group $G$ having an orbit of codimension one.
We denote by $\tilde{\alpha}\colon G \to \text{Diff}\left(M\right) $ the Lie group homomorphism induced by the action.
\end{definition}

From now on, let us assume that $M$ is a simply connected cohomogeneity one manifold, and $G$ is connected. 
By the compactness of $G$, the action $\alpha$ is proper and there exists a $G$-invariant  Riemannian metric $g$ on $M$; this is equivalent to saying that $G$ acts on the Riemannian manifold $\left(M,g\right)$ by isometries. 
Moreover, we assume that the action $\alpha$ is almost effective, namely $\text{ker}\, \tilde{\alpha}$ is discrete. 
As usual, we denote by $\pi\colon M\to M/G$ the canonical projection and we equip $M/G$ with the quotient topology relative to $\pi$.
By a result of B\'{e}rard Bergery \cite{BerBer}, the quotient space $M/G$ is homeomorphic to a circle or an interval. As we are assuming that $M$ is simply connected, we have that $M/G$ is homeomorphic to an interval $I$.
The inverse images of the interior points of the orbit space $M/G$ are known as \emph{principal orbits}, while the inverse images of the boundary points are called \emph{singular orbits}. We denote by $M^{\text{princ}}$ the union of all principal orbits, which is an open dense subset of $M$, and by $G_p$ the isotropy group at $p\in M$.

First, we will suppose $M$ is compact. It follows that $M/G$ is homeomorphic to the closed interval $I=[-1,1]$.
Denote by $\mathcal{O}_1$ and $\mathcal{O}_2$ the two singular orbits $\pi^{-1}\left(-1\right)$ and $\pi^{-1}\left(1\right)$, respectively, and fix $q_1 \in \mathcal{O}_1$.
By compactness of the $G$-orbits, there exists a minimizing geodesic $\gamma_{q_1}\colon [-1,1]\to M$ from $q_1$ to $\mathcal{O}_2$ which is orthogonal to every principal orbit.
We call a \emph{normal geodesic} a geodesic orthogonal to every principal orbit.
Let $\gamma\colon [-1,1]\to M$ be a normal geodesic between $\pi^{-1}\left(-1\right)$ and $\pi^{-1}\left(1\right)$; up to rescaling, we can always suppose that the orbit space $M/G$ is such that $\pi \circ \gamma =\text{Id}_{[-1,1]}$.
Then, by Kleiner's Lemma, there exists a subgroup $K$ of $G$ such that $G_{\gamma\left(t\right)}=K$ for all $t\in (-1,1)$ and $K$ is subgroup of $G_{\gamma\left(-1\right)}$ and $G_{\gamma\left(1\right)}$.

For $M$ non-compact, $M/G$ is homeomorphic either to an open interval or to an interval with a closed end.
In the former case, $M$ is a product manifold $M \cong I \times G/K$.
In the latter case, there exists exactly one singular orbit, and $M/G \cong I$ where $I=[0, L)$ and $L$ is either a positive number or $+\infty$. 
Analogously to the compact case, there exists a normal geodesic
$\gamma: [0, L) \rightarrow M$ such that $\gamma(0) \in \pi^{-1}(0)$ and we can suppose $\pi \circ \gamma =\text{Id}_{[0,L)}$. In addition, there exists a subgroup $K$ of $G$ such that $G_{\gamma(t)}=K$ for all $t \in (0,L)$ and if $H \coloneqq G_{\gamma(0)}$, $K$ is a subgroup of $H$. 

So we have that:
\begin{itemize}
\item $\pi^{-1}\left(t\right)\cong G/K$ for all $t\in \overset{\circ}{I}$,
\item $M^{\text{princ}} = \bigcup_{t\in \overset{\circ}{I}}\pi^{-1}\left(t\right)=\bigcup_{t\in \overset{\circ}{I}} G\cdot \gamma\left(t\right)$,
\item  for every $p_1, p_2\in M^{\text{princ}}$, $G\cdot p_1$ and $G\cdot p_2$ are diffeomorphic.
\end{itemize}
Therefore, up to conjugation along the orbits, when $M$ is compact we have three possible isotropy groups $H_1\coloneqq G_{\gamma\left(-1\right)}$, 
$H_2 \coloneqq G_{\gamma\left(1\right)}$ and $K\coloneqq G_{\gamma\left(t\right)}$, $t\in\left(-1,1\right)$. 
When $M$ is non-compact and has one singular orbit, instead,
we have two possible isotropy groups $H \coloneqq G_{\gamma\left(0\right)}$ and $K\coloneqq G_{\gamma\left(t\right)}$, $t\in\left(0,L\right)$.
From all of the above, we have that
\[ M^{\text{princ}} \cong \overset{\circ}{I} \times G/K, 
\] and so, by fixing a suitable global coordinate system, we can decompose the $G$-invariant metric $g$ as
\begin{equation} \label{metric}
g_{\gamma(t)}={\text{d}}t^2 + g_t,
\end{equation}
where ${\text{d}}t^2$ is the $(0,2)$-tensor corresponding to the vector field $\xi \coloneqq \gamma'\left(t\right)$ evaluated at the point $\gamma\left(t\right)$, and $g_t$ is a $G$-invariant metric on the homogeneous orbit $G\cdot \gamma\left(t\right)$ through the point $\gamma\left(t\right)\in M$. 

Now, we will assume $M$ is compact. By the density of $M^{\text{princ}}$ in $M$ and the Tube Theorem, $M$ is homotopically equivalent to
\begin{equation}
\left( G \times_{H_1} S_{\gamma\left(-1\right)}\right)\cup_{G/K}\left( G \times_{H_2} S_{\gamma\left(1\right)}\right),
\end{equation}
where the geodesic balls $ S_{\gamma\left(\pm 1\right)}\coloneqq \text{exp}\left(B_{\varepsilon^{\pm}}\left(0\right)\right)$, 
$B_{\varepsilon^{\pm}}\left(0\right)\subset T_{\gamma\left(\pm 1\right)}\left(G\cdot \gamma\left(\pm 1\right) \right)^{\perp}$, are normal slices to the singular orbits in $\gamma\left(\pm 1\right)$. 
Here, $ G \times_{H_i} S_{\gamma\left(\pm 1\right)}$ is the associated fiber bundle to the principal bundle $G \to G/H_i$ with type fiber $S_{\gamma\left(\pm 1\right)}$.
By Bochner's linearization theorem, $M$ is also homotopically equivalent to 
\begin{equation}\label{DecompCohom1}
\left( G \times_{H_1} B_{\varepsilon^{-}}\left(0\right)\right)\cup_{G/K}\left( G \times_{H_2} B_{\varepsilon^{+}}\left(0\right)\right).
\end{equation}
The isotropy groups $H_i$ act on $ B_{\varepsilon^{\pm}}\left(0\right)$ via the slice representation and,
since the boundary of the tubular neighborhood $\text{Tub}(\mathcal{O}_i) \coloneqq  G \times_{H_i} B_{\varepsilon^{\pm}}\left(0\right)$, $i=1,2$,
is identified with the principal orbit $G/K$ and the $G$-action on 
$\text{Tub}(\mathcal{O}_i) $ is identified with the $H_i$-action
on $ B_{\varepsilon^{\pm}} \left( 0 \right)$, then $H_i$ acts transitively on
the sphere $S^{l_i} \coloneqq \partial  B_{\varepsilon^{\pm}}$, $l_i>0$ still having isotropy $K$. The normal spheres $S^{l_i}$ are thus the homogeneous
spaces $H_i/K$, $i=1,2$.
The $H_i$-action on $S^{l_i}$, $i=1,2$, may be ineffective, but it is sufficient to quotient $H_i$ by the ineffective kernel to obtain an effective action: transitive effective actions of compact Lie groups on spheres were classified by Borel and are summarized in Table \ref{spheres}.
\begin{table}[H]
\begin{center}
\addtolength{\leftskip} {-2cm}
\addtolength{\rightskip}{-2cm}
\scalebox{0.70}{
\begin{tabular}{|c|c|c|c|c|c|c|c|c|c|}
\hline
\xrowht{15pt}
$H$               & $\text{SO}(n)$   & $\text{U}(n)$   & $\text{SU}(n)$   & $\text{Sp}(n)\text{Sp}(1)$   & $\text{Sp}(n)\text{U}(1)$   & $\text{Sp}(n)$   & $\text{G}_2$   & $\text{Spin}(7)$ & $\text{Spin}(9)$ \\ \hline \xrowht{15pt}
$K$               & $\text{SO}(n-1)$ & $\text{U}(n-1)$ & $\text{SU}(n-1)$ & $\text{Sp}(n-1)\text{Sp}(1)$ & $\text{Sp}(n-1)\text{U}(1)$ & $\text{Sp}(n-1)$ & $\text{SU}(3)$ & $\text{G}_2$     & $\text{Spin}(7)$ \\ \hline \xrowht{15pt}
$S^l=H / K$ & $S^{n-1}$        & \multicolumn{2}{c|}{$S^{2n-1}$}    & \multicolumn{3}{c|}{$S^{4n-1}$}                                               & $S^6$          & $S^7$            & $S^{15}$         \\ \hline
\end{tabular}}
\medskip
\caption{Transitive effective actions of compact Lie groups on spheres}
\label{spheres}
\end{center}
\end{table}
The collection of $G$ with its isotropy groups $G\supset H_1, H_2 \supset K $ is called the \emph{group diagram} of the cohomogeneity one manifold $M$.
Viceversa, let $G\supset H_1, H_2 \supset K $ be compact groups with $H_i/K \cong S^{l_i}$, $i=1,2$.
By the classification of transitive actions on spheres one has that the $H_i$-action on $S^{l_i}$ is linear and hence it can be extended
to an action on $ B_{\varepsilon^{\pm}}$ bounded by $S^{l_i}$, $i=1,2$.
Therefore, \ref{DecompCohom1} defines a cohomogeneity one manifold $M$.
Analogously, if $M$ is a non-compact cohomogeneity one manifold with one singular orbit, we define the \emph{group diagram} of $M$ to be the collection of $G$ and the isotropy groups $G \supset H \supset K$, where the homogeneous space $H/K$ will be a sphere. The converse is also true:  the group diagram defines a non-compact cohomogeneity one manifold $M$. In these cases, $M$ is homotopically equivalent to  $G\times_{H}B_{\epsilon}(0)$, where $B_{\epsilon}(0)\subseteq T_{\gamma(0)} (G\cdot \gamma(0))^{\perp}$ as before.

Let $M_i$ be cohomogeneity one manifolds with respect to the action of Lie groups $G_i$, $i=1,2$.
We say that the action of $G_1 $ on $M_1$ is equivalent to the action of $G_2$ on $M_2$ if there exists a Lie group isomorphism $\varphi\colon G_1 \to G_2$ and an equivariant diffeomorphism $f\colon M_1 \to M_2$ with respect to the isomorphism $\varphi$.
We shall study cohomogeneity one manifolds up to this type of equivalence. 

Moreover, if a cohomogeneity one manifold $M$ has group diagram $G\supset H_1, H_2 \supset K$ or $G\supset H \supset K$, one can show that any of the following operations results in a $G$-equivariantly diffeomorphic manifold:
\begin{enumerate}
\item switching $H_1$ and $H_2$,
\item conjugating each group in the diagram by the same element of $G$,
\item replacing $H_i$ (respectively $H$) with $aH_ia^{-1}$ (respectively $aHa^{-1}$) for $a\in N(K)_0$, where $N(K)_0$ is the identity component of the normalizer of $K$.
\end{enumerate}

\subsection{SU(3)-structures} 

An $\text{SU}(3)$-structure on a  six-dimensional differentiable manifold $M$ is the data of a Riemannian metric $g$, a $g$-orthogonal almost complex structure $J$, and a $(3,0)$-form of nonzero constant norm $\Psi=\psi_++i\psi_-$ satisfying  the 
normalization condition $\psi_+ \wedge \psi_-=\frac{2}{3}\omega^3$. 

Following a result obtained in \cite{Reichel} and later reformulated in \cite[Section 2]{Hitchin}, one can show that giving an $\text{SU}(3)$-structure is equivalent to giving a  pair of differential forms $\left(\omega, \psi_+ \right)\in \Lambda^2\left(M\right)\times \Lambda^3\left(M\right)$ satisfying suitable conditions. Here $\Lambda^k\left(M\right)$ denotes the space of differential forms of degree $k$ on $M$.
To see this, let us briefly recall the concept of stability in the context of vector spaces.

Let $V$ be a real six-dimensional vector space and let $\alpha$ be a $k$-form on $V$. 
We say that $\alpha$ is \emph{stable} if its orbit under the action of $\text{GL}(V)$ is open in $\Lambda^k\left(V^*\right)$.
Fix a volume form $\Omega \in\Lambda^6(V^*)$ on $V$ and consider the isomorphism $A\colon \Lambda^5(V^*)\to V\otimes\Lambda^6(V^*)$ defined for any $\alpha\in \Lambda^5(V^*)$ by $A(\alpha)=v\otimes \Omega$, where $v\in V$ is the unique vector such that $\iota_v \Omega=\alpha$; here, $\iota_v \Omega$ is the contraction of $\Omega$ by the vector $v$.
Fix a $3$-form $\psi\in \Lambda^3(V^*)$ and define 
\begin{align*}
K_\psi\colon  V &\to V\otimes \Lambda^6(V^*), \\
v  &\mapsto A(\iota_v \psi \wedge \psi)
\end{align*}
and 
\begin{align*}
P\colon \Lambda^3(V^*) & \to \Lambda^6(V^*)^{\otimes 2}, \\
 \psi & \mapsto \dfrac{1}{6} \text{tr}(K_\psi^2).
\end{align*}
Finally, we define the function $\lambda\colon \Lambda^3(V^*)\to \mathbb{R}$, $\lambda(\psi)=\iota_{\Omega\otimes \Omega}P(\psi)$. Note that $A$ and the sign of $\lambda$ do not depend on the choice of orientation.

\begin{proposition}[\cite{Hitchin, Reichel}]
Let $V$ be an oriented, six-dimensional real vector space. Then,
\begin{enumerate}[label=\text{\normalfont (\arabic*)}]
\item a 2-form $\omega\in\Lambda^2(V^*)$ is stable if and only if it is non-degenerate, i.e., $\omega^3\neq0$,
\item a 3-form $\psi\in\Lambda^3(V^*)$ is stable if and only if $\lambda(\psi)\neq 0$. 
\end{enumerate}
\end{proposition}

We denote by $\Lambda^3_+(V^*)$ the open orbit of stable 3-forms satisfying $\lambda(\psi)<0$. 
The $\text{GL}_+(V)$-stabilizer of a 3-form lying in this orbit is isomorphic to $\text{SL}(3,\mathbb{C})$. 
As a consequence, every $\psi\in\Lambda^3_{+}(V^*)$ gives rise to a complex structure  
\[
J_{\psi}\colon V \to V,\quad J_{\psi} \coloneqq - \frac{1}{\sqrt{|P(\psi)|}}\,K_{\psi},
\]   
which depends only on $\psi$ and on the volume form $\Omega$.  
Moreover, the complex form $\psi + i J_{\psi} \psi$ is of type $(3,0)$ with respect to $J_{\psi}$, and the real 3-form $J_{\psi} \psi$ is stable, too.

We say that a $k$-form $\alpha \in \Lambda^k(M)$ is stable if $\alpha_p$ is a stable form on the vector space $T_pM$, for all $p \in M$.
Let $\left(\omega,\psi_+\right)\in\Lambda^2(M)\times \Lambda^3_+(M)$ be a pair of stable forms on $M$ satisfying the compatibility condition $\omega\wedge\psi_+=0$ and $\lambda(\psi_+)<0$.
Consider the almost complex structure $J=J_{\psi_+}$ determined by $\psi_+$ and the volume form $\frac{\omega^3}{6}$. 
Then, the 3-form $\psi_+$ is the real part of a nowhere-vanishing $(3,0)$-form $\Psi \coloneqq \psi_+ +i \psi_-$  
with $\psi_- \coloneqq J\psi_+ = \psi_+(J\cdot,J\cdot,J\cdot) =  -\psi_+(J\cdot,\cdot,\cdot)$,  
where the last identity holds since $\psi_+$ is of type $(3,0)+(0,3)$ with respect to $J$. Moreover, 
$\omega$ is of type $(1,1)$ and, as a consequence, the $(0,2)$-tensor $g \coloneqq \omega(\cdot,J\cdot)$ is symmetric. 
Under these assumptions, the pair $(\omega,\psi_+)$ defines an $\text{SU}(3)$-structure $(g,J,\Psi)$ on $M$ provided that $g$ is a Riemannian metric and the \emph{normalization condition}
$
\psi_+ \wedge \psi_-=\frac{2}{3} \omega^3=4 \, {\text{d}}V_g
$
is satisfied, ${\text{d}}V_g$ being the Riemannian volume form. 
Conversely, given an $\text{SU}(3)$-structure $(g,J,\Psi)$ on $M$, the pair $(\omega,\psi_+)$ given by \[\omega\coloneqq g(J\cdot,\cdot), \quad \psi_+\coloneqq \text{Re}(\Psi)\] satisfies the compatibility  condition $\omega\wedge \psi_+=0$ and the stability condition $\lambda(\psi_+)<0$.

\subsubsection{Balanced $\text{\normalfont SU}(3)$-structures}

Following \cite{SU2}, we call an  $\text{SU}(3)$-structure $\left(g,J,\Psi\right)$ on a $6$-manifold $M$ \emph{balanced} if:
\begin{itemize}
\item $J$ is integrable, i.e., $\left(M,J\right)$ is a complex manifold. 
We recall that for $\text{SU}(3)$-structures, the integrability of $J$ is equivalent to requiring $({\text{d}}\Psi)^{2,2}=0$,
\item $\Psi$ is a holomorphic $(3,0)$-form,
\item ${\text{d}}\omega^{2}=0$, $\omega$ being the fundamental form.
\end{itemize}
In particular, we are interested in the non-K\"ahlerian case, i.e., ${\text{d}}\omega\neq 0$.

\begin{remark}
We can equivalently say that an $\text{SU}(3)$-structure $\left(g,J,\Psi\right)$ on $M$  is balanced if and only if 
\[
\begin{cases}
{\text{d}}\Psi=0, \\
{\text{d}}\omega^{2}=0, 
\end{cases}
\]
since
${\text{d}}\Psi=0$ if and only if $\Psi=\psi_+ + i \psi_- $ is holomorphic and the induced almost complex structure $J= J_{\psi_+}$ is integrable. 
\end{remark}

\begin{remark}
From the formulas in \cite{BedVezzSu3}, we have that if $\left(g,J,\Psi\right)$ is a balanced non-K\"ahler $\text{SU}(3)$-structure on a six-dimensional differentiable manifold $M$, $\text{Scal}(g)< 0$, $\text{Scal}(g)$ being the scalar curvature associated with the metric $g$. 
\end{remark}

\section{Balanced $\text{SU}(3)$-structures on six-dimensional cohomogeneity one manifolds} \label{sec_balanced}
Let $\left(g,J,\Psi\right)$ be an $\text{SU}(3)$-structure on a 
simply connected cohomogeneity one manifold $M$ of complex dimension $3$ for the almost effective action of a compact connected Lie group $G$. 
We are thus requiring $G$ to preserve the $\text{SU}(3)$-structure on $M$.
For the convenience of the reader, recall that
\begin{itemize}
\item $G$ preserves the metric $g$ if and only if $\tilde{\alpha}\left(h\right)$ is an isometry for each $h\in G$,
\item $G$ preserves the almost complex structure $J$ if and only if $J$ commutes with the differential $ {\text{d}}\tilde{\alpha}\left(h\right)$ for each $h \in G$,
\item $G$ preserves the $3$-form $\Psi$ if and only if $\tilde{\alpha}\left(h\right)^*\Psi=\Psi$, for each $h\in G$.
\end{itemize}
 This in particular implies that the principal isotropy $K$ acts on $T_p M$ preserving $\left(g_p,J_p,\Psi_p\right)$ for any $p\in M$, which means that $K$ is a subgroup of $\text{SU}(3)$.
 Now, since the $J$-invariant $K$-action fixes the subspace $\left<\xi|_p\right>$ of $T_pM$, then it fixes $\left<{J\xi}|_p\right>$ as well.
Let us write $T_p M $ as
\[T_p M= \left<\xi|_p\right>\oplus \left<{J \xi}|_p\right> \oplus V,
\] 
where $V$ is the four-dimensional $g_p$-orthogonal complement of $\left< \xi|_p,J\xi|_p\right>$ in $T_p M$.
Notice that $V$ is $J_p$- and $K$-invariant. 
To see the $K$-invariance, let $h \in K$ and $v \in V$. 
Then, if $ {\text{d}}\tilde{\alpha}(h)_p \left(v \right) = \lambda \left(J \xi|_p \right)+w$, for some $\lambda \in \mathbb{R}$, $w\in V$,
we would have $J \left( {\text{d}} \tilde{\alpha}(h)_p \left( v \right) \right)={\text{d}}\tilde{\alpha}(h)_p \left( J_pv \right)=-\lambda\,\xi|_p+J_p w$, which is a contradiction since the $K$-action is closed along the $G$-orbits.
Therefore, for each $h\in K$, its action on $T_p M$ is described by a $6 \times 6$ block matrix
\[
\left(
\begin{array}{c|c}
\begin{matrix} 1 & 0 \\ 0 & 1 \end{matrix} & \phantom{\begin{matrix} 0 & 0 & 0 & 0 \\ 0 & 0 & 0 & 0 \end{matrix}}\\
\hline \phantom{\begin{matrix} 0 & 0 \\ 0 & 0 \\ 0 & 0 \\ 0 & 0 \end{matrix}} & A 
\end{array}
\right)
\]
with respect to the decomposition of $T_p M= \left<\xi|_p\right>\oplus \left<{J \xi}|_p\right> \oplus V$.
Since the matrix above is in $\text{SU}(3)$, we have $A \in \text{SU}(2)$ and hence $K$ can be identified with a subgroup of $\text{SU}(2)$.
Therefore, $ \mathfrak{k}\coloneqq \text{Lie}\left(K\right)$ is $\{0\}, \, \mathbb{R}$, or $\mathfrak{su}\left(2\right)$. 
As observed in \cite{Pod}, all the possible candidate pairs $\left(\mathfrak{g},\mathfrak{k}\right)$, with $\mathfrak{g}$ compact, which may admit an $\text{SU}(3)$-structure in cohomogeneity one are: 
\begin{enumerate}
\item[(a)] if $\mathfrak{k}=\{0\}$, then 
\begin{enumerate}
\item[(1)] $\mathfrak{g}=\mathfrak{su}\left(2\right) \oplus \mathbb{R} \oplus \mathbb{R} $,
\item[(2)]  $\mathfrak{g}= \underbrace{\mathbb{R} \oplus \ldots \oplus \mathbb{R}}_{\text{5 times}}$,
\end{enumerate}
\item[(b)] if $\mathfrak{k}=\mathbb{R}$, then 
\begin{enumerate}
\item[(1)] $\mathfrak{g}=\mathfrak{su}\left(2\right) \oplus \mathfrak{su}\left(2\right)$,
\item[(2)] $ \mathfrak{g}= \mathfrak{su}\left(2\right) \oplus \mathbb{R} \oplus \mathbb{R} \oplus \mathbb{R} $,
\item[(3)] $\mathfrak{g}= \underbrace{\mathbb{R} \oplus \ldots \oplus \mathbb{R}}_{\text{6 times}}$,
\end{enumerate}  
\item[(c)] if $\mathfrak{k}=\mathfrak{su}\left(2\right)$, then
\begin{enumerate}
\item[(1)] $\mathfrak{g}=\mathfrak{su}\left(2\right) \oplus \mathfrak{su}\left(2\right) \oplus \mathbb{R} \oplus \mathbb{R}$,
\item[(2)] $\mathfrak{g}= \mathfrak{su}\left(2\right) \oplus \underbrace{\mathbb{R} \oplus \ldots \oplus \mathbb{R}}_{\text{5 times}}$,
\item[(3)] $\mathfrak{g}= \mathfrak{su}\left(3\right)$. 
\end{enumerate}
\end{enumerate} 

Under the assumption of simply connectedness of $M$, we can discard some pairs of this list.
If $M$ is compact, by Hoelcher's classification \cite[Proposition 3.1]{Hoelscher},  we can readily discard cases (a.2), (b.2), (b.3), (c.1), and (c.2).
For the case when $M$ is non-compact and has one singular orbit, we can suitably adapt \cite[Proposition 1.8]{Hoelscher} which deals with the compact case, to obtain:
\begin{proposition}
Let $M$ be the non-compact cohomogeneity one manifold given by the group diagram $G \supset H \supset K$ with $H/K = S^l$, and $l \geq 1$. 
Then, $\pi_1(M) \cong \pi_1 (G/H)$.
In particular $M$ is simply connected if and only if the image of $\pi_1 (H/K) = \pi_1 (S^l)$ generates $\pi_1(G/K)$ under the natural inclusions.
\end{proposition}

We know that $\pi_1 (S^l)$ is either $\{ 0 \}$ or $\mathbb{Z}$.
Now, we observe that for cases (a.1) and (c.1), $\pi_1 (G/K)= \mathbb{Z}^2$, for cases (a.2), (b.3) and (c.2), $\pi_1 (G/K)= \mathbb{Z}^5$, and for case (b.2), $\pi_1 (G/K)$ is either $\mathbb{Z}^2$ or $\mathbb{Z}^3$. 
If $M$ is non-compact and has no singular orbits, $\pi_1 (M)=\pi_1 (G/K)$.
Hence, when $M$ is non-compact, we can discard the pairs (a.1), (a.2),  (b.2), (b.3), (c.1) and (c.2) as $\pi_1(M)$ would be infinite.
Therefore, the possible pairs which may admit a balanced $\text{SU}(3)$-structure on a simply connected manifold of cohomogeneity one under the almost effective action of a compact connected Lie group $G$ are (a.1) (only when M is compact), (b.1), and (c.3).

\begin{remark}\label{pq}
In case (b.1), we shall need to divide the discussion depending on the embeddings of $\mathfrak{k}=\mathbb{R}$ in $\mathfrak{g}=\mathfrak{su}(2)\oplus \mathfrak{su}(2)$ which, up to isomorphism, are all generated by an element of the form
\[\begin{pmatrix}
ip & 0 & 0 & 0 \\ 
0 & -ip & 0 & 0 \\ 
0 & 0 & iq & 0 \\ 
0 & 0 & 0 & -iq
\end{pmatrix} \in \mathfrak{su}(2)\oplus \mathfrak{su}(2),
\]
\\
with fixed $p, q \in \mathbb{N}$. 
We can assume either $(p,q)=(1,0)$ or $p,q$ to be coprime if neither is zero.
Notice that when $(p,q)=(1,1)$ or $(p,q)=(1,0)$, $\mathfrak{k}$ induces a decomposition of $\mathfrak{g}$ into $\text{Ad}(K)$-modules, some of which are equivalent. 
In the former case, we shall say that $\mathfrak{k}$ is diagonally embedded in $\mathfrak{g}$, while in the latter $\mathfrak{k}$ is said to be trivially embedded in one of the two $\mathfrak{su}(2)$-factors of $\mathfrak{g}$. 
When instead $p, q$ are different and nonzero, the $\text{Ad}(K)$-modules are pairwise inequivalent. 
\end{remark}

From now on, for each $p\in M^{\text{princ}}$, let $\mathfrak{m}_p\eqqcolon \mathfrak{m}$ be an $\text{Ad}(K)$-invariant complement of $\mathfrak{k}$ in $\mathfrak{g}$. For each $p\in M^{\text{princ}}$, we have that $T_p M=\left<\xi|_p\right>\oplus \widehat{\mathfrak{m}}|_p$, where for every $X \in \mathfrak{g}$, we denote by $\widehat X$ the action field
\[
\widehat X_p = \frac{{\text{d}}}{{\text{d}}t}\Big\rvert_{t=0}(\exp tX)\cdot p,\quad p \in M.
\]
It is known that since $M^{\text{princ}}\cong \overset{\circ}{I} \times G/K $, every $G$-invariant structure on $M^{\text{princ}}$ can be expressed via a $K$-invariant structure on $\left<\xi\right>\oplus \widehat{\mathfrak{m}} $, with $C^{\infty}(\overset{\circ}{I})$-coefficients. 
Let $\mathfrak{m}=\mathfrak{m}_1 \oplus \ldots \oplus \mathfrak{m}_r$ be the decomposition of $\mathfrak{m}$ into irreducible $\text{Ad}(K)$-modules. Recall that if the $\mathfrak{m}_i$'s are pairwise inequivalent, then they are orthogonal with respect the metric $g_t$,  for every $t$ (see \ref{metric}).
Otherwise, the expression of the metric strongly depends on the specific equivalence of the modules. In all cases, we recover the whole $\text{SU}(3)$-structure from a pair of $G$-invariant stable forms $\left(\omega,\psi_+\right)$ of degree two and three, respectively.

To fix the notations, in what follows, we shall denote by 
\begin{itemize}
\item $\mathcal{B}$ the negative of the Killing--Cartan form on $\mathfrak{g}$,
\item $\left\{\tilde{e}_i\right\}_{i=1,2,3}$ the standard basis for $\mathfrak{su}(2)$ given by 
\[
\tilde{e}_1=\begin{pmatrix} i & 0 \\ 0 & -i \end{pmatrix}, \qquad \tilde{e}_2=\begin{pmatrix} 0 & i \\ i & 0 \end{pmatrix}, \qquad \tilde{e}_3=\begin{pmatrix} 0 & 1 \\ -1 & 0 \end{pmatrix},
\]
\item $\left\{f_i\right\}_{i=1,\ldots, m}$ the standard basis for $\mathfrak{g}=\mathfrak{k}\oplus \mathfrak{m}$, $\mathfrak{k}=\left<f_1,\ldots,f_k\right>$, $\mathfrak{m}=\left<f_{k+1},\ldots,f_m\right>$, where $k=\operatorname{dim}\mathfrak{k}$,
$m=\operatorname{dim}\mathfrak{g}$,
\item $e_1 \coloneqq \xi \cong \frac{\partial}{\partial t}$,
\item $e_i \coloneqq \widehat{f}_{j}$, $j=\operatorname{dim}\mathfrak{k}-1+i$, the Killing vector fields on $M^{\text{princ}}$ induced by the $G$-action, for $i=2,\ldots,6$,
\item $e^i$ the dual 1-forms to $e_i$.
\end{itemize}
Therefore, in what follows $\left\{e_i \right\}_{i=1,\ldots,6}$ will be vectors on $M^{\text{princ}}$ which provide a basis
for $T_pM$ at each point $p=\gamma(t)\in M^{\text{princ}}$, where $\gamma\colon \overset{\circ}{I} \to M$ is a normal geodesic through the point $p$.

Moreover, we recall some basic facts about $G$-actions which will be useful for our discussion:
\begin{itemize}
\item Since $g \cdot \gamma_p=\gamma_{g\cdot p } $ for the uniqueness of the normal geodesic $\gamma$ starting from the point $g\cdot p $, we have that $\Phi_1^{\hat{X}}\circ \Phi_t^{\xi}(p)=\Phi_t^{\xi}\circ \Phi_1^{\hat{X}}(p)$, where $\Phi_t^v$ denotes the flow of the vector field $v$ evaluated at time $t$. This is equivalent to $[\xi, \hat{X}]=0$, for each $X\in \mathfrak{g}$;
\item A $k$-form $\alpha$ on $M^\text{princ}$ of the form
\[
\alpha=\sum_{i_1<\ldots<i_k=1}^{6} a_{i_1\ldots i_k} \,e^{i_1\ldots i_k},
\]	
with $a_{i_1\ldots i_k} \in C^{\infty}(\overset{\circ}{I})$ for all $i_1<\ldots<i_k$, is $G$-invariant if and only if $\alpha_p$ is $K$-invariant for all $p\in M^{\text{princ}}$. Here $e^{i_1\ldots i_k}$ is a shorthand for the wedge product $e^{i_1}\wedge \ldots \wedge e^{i_k}$ of $1$-forms.
Analogously we shall indicate with $\beta^k$ the wedge product of $\beta$ with itself for $k$-times $\beta \wedge \ldots \wedge \beta$;
\item If $\alpha$ is a $G$-invariant $k$-form on $M$ and $v_1,\ldots, v_k$ are $G$-invariant vector fields on $M$, then $\alpha\left(v_1,\ldots,v_n\right)|_p$ is constant along the $G$-orbit through $p$, for each $p\in M$.
\end{itemize}

We are now ready to state Theorem A. 

\begin{theoremA}
Let $M$ be a six-dimensional simply connected cohomogeneity one manifold under the almost effective action of a connected Lie group $G$, and let $K$ be the principal isotropy group. Then, the principal part $M^{\text{\normalfont princ}}$ admits a $G$-invariant balanced non-K\"ahler $\text{\normalfont SU}(3)$-structure $\left(g,J,\Psi\right)$ if and only if 
$\left( \mathfrak{g}, \mathfrak{k} \right)=\left(\mathfrak{su}(2)\oplus\mathfrak{su}(2), \Delta \mathbb{R} \right)$ 
or
$M$ is compact and $\left( \mathfrak{g}, \mathfrak{k} \right)=\left(\mathfrak{su}(2)\oplus 2\mathbb{R},\{0\} \right)$.
\end{theoremA}

\section{Proof of Theorem A} 

From all the above discussion and the previous lemmas, the only possible pairs allowing $M^{\text{princ}}$ to support a balanced $\text{SU}(3)$-structure are (a.1) with $M$ compact, (c.3), and (b.1). We investigate these three cases separately.

For any of these cases, we shall consider the generic pair $(\omega,\psi_+)$
	 of $G$-invariant forms on $M^{\text{princ}}$ of degree two and three, respectively, with $C^{\infty}(\overset{\circ}{I})$-coefficients.
	In order for the pair $(\omega,\psi_+)$ to define a $G$-invariant balanced non-K\"ahler $\text{SU}(3)$-structure on $M^{\text{princ}}$, we have to impose the following conditions:
	\begin{enumerate}[ref=(\arabic*)]
		\item \label{1} the stability conditions:
		\begin{itemize}
			\item $\omega^3 \neq 0$,
			\item $\lambda \coloneqq \lambda\left(\psi_+\right)<0$,
		\end{itemize}
		\item the compatibility conditions $\psi_{\pm} \wedge \omega=0$,
		\item the normalization condition $\psi_+ \wedge \psi_-=\frac{2}{3} \omega^3$,
		\item ${\text{d}} \psi_{\pm}=0$,
		\item the balanced condition ${\text{d}}\omega^2=0$,
		\item the non-K\"ahler condition ${\text{d}}\omega \neq 0$,
		\item \label{7} the positive-definiteness of the induced symmetric bilinear form $g \coloneqq \omega(\cdot, J \cdot)$ on $M^{\text{princ}}$.
\end{enumerate}

We start by case (b.1).

\subsection{Case (b.1)}\label{Case (b.1)}
$\left(\mathfrak{g},\mathfrak{k}\right)=\left(\mathfrak{su}(2)\oplus \mathfrak{su}(2), \mathbb{R}\right)$.

In the notation of Remark \ref{pq}, let us suppose $p, q$ nonzero and coprime with $(p,q) \neq (1,1)$, first.
Consider the $\mathcal{B}$-orthonormal basis of $\mathfrak{g}$ given by
\begin{equation}\label{basepq}
\begin{alignedat}{3}
&f_1=\dfrac{1}{2\sqrt{2(p^2+q^2)}}\begin{pmatrix} p \tilde{e}_1 & 0 \\ 0 & q\tilde{e}_1 \end{pmatrix} \hspace{0.75em}
&&f_2=\dfrac{1}{2\sqrt{2(p^2+q^2)}}\begin{pmatrix} q \tilde{e}_1 & 0 \\ 0 & -p\tilde{e}_1 \end{pmatrix} \\
&f_3=\dfrac{1}{2\sqrt{2}}\begin{pmatrix} \tilde{e}_3 & 0 \\ 0 & 0 \end{pmatrix}  \hspace{0.75em}
&&f_4=\dfrac{1}{2\sqrt{2}}\begin{pmatrix} 0 & 0 \\ 0 & \tilde{e}_3 \end{pmatrix} \\
&f_5=\dfrac{1}{2\sqrt{2}}\begin{pmatrix} \tilde{e}_2 & 0 \\ 0 & 0 \end{pmatrix} \hspace{0.75em}
&&f_6=\dfrac{1}{2\sqrt{2}}\begin{pmatrix} 0 & 0 \\ 0 & \tilde{e}_2 \end{pmatrix}.
\end{alignedat}
\end{equation}
and take $\mathfrak{k}=\left<f_1\right>$. Notice that since rk$(\mathfrak{su}(2))=1$, this assumption is not restrictive.  The decomposition of $\mathfrak{g}$ into irreducible $\text{Ad}\left(K\right)$-modules is given by 
\[\mathfrak{g}=\mathfrak{k}\oplus \mathfrak{a}\oplus \mathfrak{b}_1\oplus \mathfrak{b}_2,
\]
where $\mathfrak{a}\coloneqq \left<f_2\right>$ is $\text{Ad}(K)$-fixed, $\mathfrak{b}_1\coloneqq \left<f_3, f_5\right>$ and $\mathfrak{b}_2\coloneqq \left<f_4, f_6\right>$, hence $\mathfrak{m}= \mathfrak{a}\oplus \mathfrak{b}_1\oplus \mathfrak{b}_2$.
Fix the orientation given by $\Omega=e^{1\ldots 6}$ and consider the general $G$-invariant $3$-form $\psi_+$ on $M^{\text{princ}}$,
\[
\psi_+ \coloneqq p_1\, e^{135} + p_2\, e^{146}+ p_3\, e^{235} + p_4\, e^{246},
\]
where $p_j\in C^{\infty}(\overset{\circ}{I})$, $j=1,\ldots, 4$.
A simple calculation shows that the stability condition $\lambda(\psi_+)<0$ never holds, since
$\lambda(\psi_+)=\left(p_1 p_4 - p_2 p_3 \right)^2\geq 0$.
Alternatively, by \cite[Proposition 2]{Hitchin} we can directly see that $\lambda$ is non-negative, 
as $\psi_+ = (p_1\, e^1 +p_3\, e^2) \wedge e^{35}+ (p_2\, e^1 +p_4\, e^2 ) \wedge e^{46}$ can be written as a sum of real decomposable forms. 

Now let $(p,q)=(1,0)$ and consider the $\mathcal{B}$-orthogonal basis of $\mathfrak{g}$ given by \ref{basepq} when $(p,q)=(1,0)$ and assume $\mathfrak{k}=\left<f_1\right>$ as before.
Then, the decomposition of $\mathfrak{g}$ into irreducible $\text{Ad}(K)$-modules is given by 
\[
\mathfrak{g}=\mathfrak{k}\oplus \mathfrak{b}_1\oplus \mathfrak{a}_1\oplus \mathfrak{a}_2 \oplus \mathfrak{a}_3,
\]
where $\mathfrak{b}_1 \coloneqq \left<f_3, f_5\right>$, $\mathfrak{a}_1\coloneqq \left<f_2\right>$, $\mathfrak{a}_2\coloneqq \left<f_4\right>$ and $\mathfrak{a}_3\coloneqq \left<f_6\right>$. Observe that the $\mathfrak{a}_i$'s are equivalent.
Consider the generic $G$-invariant $3$-form $\psi_+$ on $M^{\text{princ}}$, which is of the form
\[
\psi_+ \coloneqq p_1\, e^{124} + p_2\, e^{126}+ p_3\, e^{135} + p_4\, e^{146}+ p_5\, e^{235} + p_6\, e^{246}+ p_7\, e^{345} + p_8\, e^{356},
\]
where $p_j\in C^{\infty}(\overset{\circ}{I})$, $j=1,\ldots,8$.
It is straightforward to show that 
$\lambda(\psi_+)=\left(p_1 p_8 + p_2 p_7 - p_3 p_6 +p_4 p_5\right)^2\geq 0$.

\begin{remark}
By the previous discussion we have that when $\left(\mathfrak{g},\mathfrak{k}\right)=\left(\mathfrak{su}(2)\oplus \mathfrak{su}(2), \mathbb{R}\right)$ with $\mathfrak{k}$ not diagonally embedded in $\mathfrak{g}$, $M$ admits no $G$-invariant $\text{\normalfont SL}(3,\mathbb{C})$-structures, i.e., $G$-invariant stable $3$-forms inducing an almost complex structure on $M$.
\end{remark}

Finally, let us consider the case when $\mathfrak{k}$ is diagonally embedded in $\mathfrak{g}$.
Without loss of generality, we can assume $(p,q)=(1,1)$.
We consider the $\mathcal{B}$-orthonormal basis of $\mathfrak{g}$ given by  \ref{basepq} when $(p,q)=(1,1)$.
The decomposition of $\mathfrak{g}$ into irreducible $\text{Ad}\left(K\right)$-modules is given by 
\[\mathfrak{g}=\mathfrak{k}\oplus \mathfrak{a}\oplus \mathfrak{b}_1\oplus \mathfrak{b}_2,
\]
where $\mathfrak{k}=\left<f_1\right>$, $\mathfrak{a}\coloneqq \left<f_2\right>$ is $\text{Ad}(K)$-fixed, $\mathfrak{b}_1\coloneqq \left<f_3, f_5\right>$ and $\mathfrak{b}_2\coloneqq \left<f_4, f_6\right>$.
Then, $\mathfrak{m}= \mathfrak{a}\oplus \mathfrak{b}_1\oplus \mathfrak{b}_2$.
Unlike the case $p\neq q$ both nonzero, here the equivalence of the $\mathfrak{b}_i$-modules implies that the metric $g$ on $M^{\text{princ}}$ is not necessarily diagonal but of the form:
\[
g={\text{d}}t^2 + f(t)^2 \mathcal{B}|_{\mathfrak{a}\times\mathfrak{a}}+ h_1(t)^2 \mathcal{B}|_{\mathfrak{b}_1\times \mathfrak{b}_1}
+ h_2(t)^2 \mathcal{B}|_{\mathfrak{b}_2\times \mathfrak{b}_2}+ \mathcal{Q}|_{\mathfrak{b}_1\times \mathfrak{b}_2},
\]
for some $f, h_1, h_2 \in C^{\infty}(\overset{\circ}{I})$, where $\mathcal{Q}$ denotes a symmetric quadratic form on the isotypic component $\mathfrak{b}_1\oplus \mathfrak{b}_2$.
In particular, the metric coefficients $g_{ij}\coloneqq g(e_i,e_j)$ must satisfy
\begin{equation}\label{condmetric}
\begin{aligned}
g_{1i} &=g_{i1}=0, \quad i=2, \ldots, 6, \\
g_{2i} &=g_{i2}=0, \quad i=3,\ldots, 6, \\
g_{33}&=g_{55}, \quad g_{35}=g_{53}=0, \\
g_{44}&=g_{66}, \quad g_{46}=g_{64}=0.
\end{aligned}
\end{equation}
where $e_i$, $i=1,\ldots,6$, are the vector fields defined in the usual way.
Fix the orientation given by $\Omega \coloneqq e^{1\ldots6}$, and consider a pair of $G$-invariant forms $\left(\omega,\psi_+\right)$ of degree two and three, given respectively by
\begin{align*}
\omega \coloneqq&  h_1\, e^{12}+h_2\, e^{35} + h_3\, e^{46}+h_4 (e^{34}+e^{56}) + h_5 (e^{36}+e^{45} ), \\
\psi_+ \coloneqq&  p_1\, e^{135} + p_2\, e^{146}+ p_3 (e^{134}+e^{156}) + p_4 (e^{136}+e^{145})\\
&+ p_5 \,e^{235} + p_6\, e^{246}+ p_7 ( e^{234} +e^{256})+ p_8 (e^{236}+e^{245}),
\end{align*}
where $h_i, p_j\in C^{\infty}(\overset{\circ}{I})$, $i=1,\ldots,5$, $j=1,\ldots,8$.
Moreover, the structure equations are given by
\[
{\text{d}}e^1=0, ~ {\text{d}}e^2=\frac{1}{2} \left( e^{35} -e^{46}\right), ~ {\text{d}}e^3=- \frac{1}{2} e^{25}, ~ {\text{d}}e^4=\frac{1}{2}e^{26}, ~
{\text{d}}e^5=\frac{1}{2}e^{23}, ~ {\text{d}}e^6=- \frac{1}{2}e^{24}.
\]
In order to find a $G$-invariant balanced non-K\"ahler $\text{SU}(3)$-structure on $M^{\text{princ}}$, we have to impose the conditions \ref{1} to \ref{7}  listed at the beginning of this Section, together with \ref{condmetric}.
We write all conditions in terms of the coefficients $h_i, p_j$ of $\left(\omega,\psi_+\right)$, for $i=1,\ldots,5$, $j=1,\ldots,8$. 

(1) The first stability condition $\omega^3 \neq 0$ is:
$$
\omega^3 =-6 h_1 (h_2 h_3 -h_4^2 -h_5^2) e^{123456} \neq 0.
$$
In particular, $h_1 \neq 0$. 
For the second stability condition $\lambda <0$ we compute:
$$
\begin{array}{ll}
\lambda =
&p_2^2 p_5^2 - 2 p_1 p_2 p_5 p_6 + 4 p_3^2 p_5 p_6 + 4 p_4^2 p_5 p_6 + p_1^2 p_6^2 \\
&- 4 p_2 p_3 p_5 p_7 - 4 p_1 p_3 p_6 p_7 + 4 p_1 p_2 p_7^2 - 4 p_4^2 p_7^2 \\
&- 4 p_2 p_4 p_5 p_8 - 4 p_1 p_4 p_6 p_8 + 8 p_3 p_4 p_7 p_8 + 4 p_1 p_2 p_8^2 \\
&- 4 p_3^2 p_8^2
\end{array}
$$

(2) Compatibility conditions: 
\begin{equation*}
\label{eq:compp}
\begin{cases}
-h_3 p_1 - h_2 p_2 + 2 h_4 p_3 + 2 h_5 p_4 =0, \\
-h_3 p_5 - h_2 p_6 + 2 h_4 p_7 + 2 h_5 p_8 =0, \\
-h_3 q_1 - h_2 q_2 + 2 h_4 q_3 + 2 h_5 q_4 =0, \\
-h_3 q_5 - h_2 q_6 + 2 h_4 q_7 + 2 h_5 q_8 =0.
\end{cases}
\end{equation*}

(3) Normalization condition $\psi_+ \wedge \psi_- = \dfrac{2}{3} \omega^3$:
\begin{equation*}
-6 \sqrt{- \lambda}=-4 h_1 (h_2 h_3 -h_4^2 -h_5^2).
\end{equation*}

(4) A computation shows ${\text{d}}\psi_+=0$ if and only if
\begin{equation}\label{chiusura}
\begin{cases}
p_8'-p_3=0, \\
p_7'+p_4=0, \\
p_5=p_6, \\
p_6'=0.
\end{cases}
\end{equation}
Let us suppose that $\psi_+$ is stable with $\lambda<0$, and consider the induced almost complex structure $J$ on $M^{\text{princ}}$. 
Recall that by $G$-invariance,
$\psi_-=J \psi_+$ needs to be of the same general form of $\psi_+$, namely
\begin{align*}
\psi_- =& q_1 e^{135} + q_2 e^{146}+ q_3 (e^{134}+e^{156}) + q_4 (e^{136}+e^{145})\\
&+q_5 e^{235} + q_6 e^{246}+ q_7 ( e^{234} +e^{256})+ q_8 (e^{236}+e^{245}),
\end{align*}
where the $q_i$'s are functions of $\{p_j\}_{j=1,\ldots,8}$ for $i=1,\ldots,8$.
Therefore, ${\text{d}}\psi_-=0$ if and only if
\begin{equation}\label{cochiusura}
\begin{cases}
q_8'-q_3=0, \\
q_7'+q_4=0, \\
q_5=q_6, \\
q_6'=0.
\end{cases}
\end{equation}

(5) Balanced condition: One has that ${\text{d}}\omega^2=0$ if and only if 
\[
\dfrac{h_1}{2}\left(h_3-h_2\right)-\left(h_2h_3-h_4^2-h_5^2\right)'=0.
\]

(6) Non-K\"ahler condition $d \omega \neq 0$: ${\text{d}}\omega=0$ if and only if 
\[
\begin{cases}
-\dfrac{h_1}{2}+h_2'=0, \\
(h_2+h_3)'=0,\\
h_4=h_5=0.
\end{cases}
\]

(7) We would also need to check that the metric $g$ is positive definite.

The only non-redundant equation from (4.2) is $g_{12}=0$, which is equivalent to 
\begin{equation}\label{condmetric1}
p_1p_6+p_2p_6-2p_3p_7-2p_4p_8=0, 
\end{equation}
where we have already assumed $p_5=p_6$ from (4.3).
Since $p_6'=0$ and all the conditions for the $G$-invariant balanced non-K\"ahler $\text{SU}(3)$-structure involve only homogeneous polynomials,
we can assume either $p_6=0$ or $p_6=1$, up to scalings.

We compute $\psi_-$ 
and note that
the condition $q_5=q_6$ implies
\begin{equation}\label{eq:q5q6}
    (p_1 -p_2)(-p_6^2 +p_7^2 +p_8^2)=0.
\end{equation}

We will show that both for $p_6=0$ and for $p_6=1$, (1)--(7) requires $p_1=p_2$.
If $p_6=0$, then 
$$
\begin{array}{ll}
\lambda 
=4p_1 p_2 (p_7^2 +p_8^2 ) -  4(p_4 p_7 - p_3 p_8)^2.
\end{array}
$$
From equation \ref{eq:q5q6} we must have $p_1=p_2$, as if it were not true, we would have $p_7=p_8 = 0$ and therefore $\lambda=0$.

Suppose now that $p_6=1$. Equation \ref{eq:q5q6} implies that we have two possibilities: either $p_7^2 +p_8^2=1$ or $p_1=p_2$.
We first assume $p_7^2 +p_8^2=1$. Differentiating this expression and using (4.3), we get
$$
p_4 p_7 -p_3 p_8 =0.
$$
Hence, 
$$
\begin{array}{ll}
\lambda
&=4 (p_3^2 +p_4^2 -(p_3 p_7 + p_4 p_8)^2 ).\\
\end{array}
$$
By Cauchy--Schwarz, 
$$
(p_3 p_7 + p_4 p_8)^2 \leq (p_3^2 +p_4^2)(p_7^2 +p_8^2),
$$
and since $p_7^2 +p_8^2=1$, we get $\lambda \geq 0$ which gives a contradiction.
Hence any solution will have $p_1=p_2$.

We now provide an example of a solution to (1)--(7) on $M^\text{princ}$. To do so, we set $p_1=p_2=p_3=p_5=p_6=p_8=h_4=h_5=0$. In particular, the compatibility conditions will then be automatically satisfied. A solution is given by setting
$$
\begin{array}{ll}
h_2=&-3t, \\
h_3=&t, \\
h_1=& -3, \\
p_4=& - \dfrac{3 \sqrt{2}}{2} t^{1/2}, \\
p_7=& \sqrt{2} t^{3/2}, \\
\end{array}
$$
for $t>0$
and all other coefficients equal to 0.
This satisfies conditions 1--7 and hence gives $G$-invariant balanced non-K\"ahler $\text{SU}(3)$-structure on $M^\text{princ}$.
Then, by performing the change of variable 
$$
\tilde{t}(t)=\sqrt{6} t^{3/2},
$$
we get
$$
\begin{array}{ll}
\omega &=-3\, e^{12} + 6^{-1/3} \tilde{t}^{2/3} ( -3 e^{35} + e^{46}), \\
\psi_+ &= - 2^{-2/3} 3^{5/6} \tilde{t}^{1/3}(e^{136}+e^{145}) + 3^{-1/2} \tilde{t} (e^{234}+e^{256}),
\end{array}
$$
and the metric on $M^\text{princ}$ with respect to the $\tilde{t}$ parameter is then represented by the matrix (so the basis of the dual space taken is $\{ d\tilde{t}, e^2, e^3, e^4, e^5, e^6 \}$)
\[
(g_{ij})=
\begin{pmatrix} 
    1 & 0 & 0 & 0 & 0 & 0 \\
    0 & \tilde{t}^2 & 0 & 0 & 0 & 0 \\
    0 & 0 & \dfrac{3}{2}\tilde{t}^2 & 0 & 0 & 0 \\
    0 & 0 & 0 & \dfrac{1}{2}\tilde{t}^2 & 0 & 0 \\
    0 & 0 & 0 & 0 & \dfrac{3}{2}\tilde{t}^2 & 0 \\
    0 & 0 & 0 & 0 & 0 & \dfrac{1}{2}\tilde{t}^2
\end{pmatrix}.
\]

It follows from \cite{Verdiani} or \cite{FH17} that the solution above does not extend to a singular orbit at $t=0$ to give a smooth metric on the whole manifold.

\subsection{Case (c.3)}\label{Case (c.3)} $\mathfrak{g}=\mathfrak{su}(3)$, $\mathfrak{k}=\mathfrak{su}(2)$.

Consider the $\mathcal{B}$-orthogonal basis of $\mathfrak{g}$ given by 
\begin{equation*}
\begin{aligned}
f_1&=\begin{pmatrix}
0 & i & 0 \\
i & 0 & 0 \\ 
0 & 0 & 0
\end{pmatrix} & f_2&=\begin{pmatrix}
0 & 1 & 0 \\ 
-1 & 0 & 0 \\ 
0 & 0 & 0
\end{pmatrix} &  f_3&=\begin{pmatrix}
i & 0 & 0 \\ 
0 & -i & 0 \\ 
0 & 0 & 0
\end{pmatrix} &
f_4&=\begin{pmatrix}
0 & 0 & i \\ 
0 & 0 & 0 \\ 
i & 0 & 0
\end{pmatrix} \\
f_5&=\begin{pmatrix}
0 & 0 & 1 \\ 
0 & 0 &d 0 \\ 
-1 & 0 & 0
\end{pmatrix} &  f_6&=\begin{pmatrix}
0 & 0 & 0 \\ 
0 & 0 & i \\ 
0 & i & 0
\end{pmatrix}
&  f_7&=\begin{pmatrix}
0 & 0 & 0 \\ 
0 & 0 & 1 \\ 
0 & -1 & 0
\end{pmatrix} &  f_8&=\frac{1}{\sqrt{3}} \begin{pmatrix}
i & 0 & 0 \\ 
0 & i & 0 \\ 
0 & 0 & -2i
\end{pmatrix}\hspace{-3pt}.
\end{aligned}
\end{equation*}
Then, $\mathfrak{k}=\left< f_1, f_2, f_3\right>$.
Let $\mathfrak{a}\coloneqq \left<f_8\right>$ and $\mathfrak{n}\coloneqq \left< f_4, f_5, f_6, f_7\right>$, hence, $\mathfrak{m}=$ \mbox{$\mathfrak{a}\oplus \mathfrak{n}$}.
Since the $\text{Ad}(K)$-invariant irreducible modules in the decomposition of $\mathfrak{g}$ are pairwise inequivalent, the metric $g$ on  $M^{\text{princ}}$ is diagonal and, in particular, it is of the form 
\[
g={\text{d}}t^2+h(t)^2\mathcal{B}|_{\mathfrak{a}\times\mathfrak{a}}+f(t)^2\mathcal{B}|_{\mathfrak{n}\times \mathfrak{n}}, 
\]
for some positive $h, f \in C^{\infty}(\overset{\circ}{I})$.
Moreover, with respect to the frame $\{e_i\}_{i=1,\ldots,6}$ of $M^{\text{princ}}$, the structure equations are given by 
\begin{align*}
& {\text{d}}e^1=0, &  &{\text{d}}e^2=-\sqrt{3}e^{36}, & &{\text{d}}e^3=\sqrt{3}e^{26}, \\ &{\text{d}}e^4=-\sqrt{3}e^{56}, & &{\text{d}}e^5=\sqrt{3}e^{46}, 
& &{\text{d}}e^6=-\sqrt{3}(e^{23}+e^{45}).
\end{align*}
Fix the volume form $\Omega=e^{1\ldots 6}$.
One can easily show that a pair of generic $G$-invariant forms $\left(\omega,\psi_+\right)$ on $M^\text{princ}$ of degree two and three is given, respectively, by 
\begin{align*}
\omega \coloneqq&  h_1\, e^{16}+h_2\, (e^{23}+e^{45}) + h_3\,(e^{24}-e^{35})+h_4 (e^{25}+e^{34} ), \\
\psi_+ \coloneqq&  p_1\, (e^{123}+e^{145}) + p_2\, (e^{124}-e^{135})+ p_3 (e^{246}-e^{356}) + p_4 (e^{236}+e^{456})\\
&+ p_5 \,(e^{125}+e^{134}) + p_6\,(e^{256}+e^{346}),
\end{align*}
where $h_i, p_j\in C^{\infty}(\overset{\circ}{I})$, $i=1,\ldots,4$, $j=1,\ldots,6$.
As we did for case (b.1), we are going to show that the system of equations resulting from imposing conditions (1)--(7) is incompatible.

A simple computation shows that ${\text{d}}\psi_+=0$ if and only if
\[
\begin{cases}
p_6'-2\sqrt{3}\,p_2=0, \\
p_3'+2\sqrt{3}\,p_5=0, \\
p_4=p_4'=0.
\end{cases}
\]
From the $G$-invariance, 
\begin{align*}
\psi_- \coloneqq&  q_1\, (e^{123}+e^{145}) + q_2\, (e^{124}-e^{135})+ q_3 (e^{246}-e^{356}) + q_4 (e^{236}+e^{456})\\
&+ q_5 \,(e^{125}+e^{134}) + q_6\,(e^{256}+e^{346}),
\end{align*}
where the $q_i$'s are functions of $\{p_j\}_{j=1,\ldots,6}$ for $i=1,\ldots,6$.
Therefore, ${\text{d}}\psi_-=0$ if and only if
\[
\begin{cases}
q_6'-2\sqrt{3}q_2=0, \\
q_3'+2\sqrt{3}q_5=0, \\
q_4=q_4'=0.
\end{cases}
\]
In particular, from $p_4=0$, it follows
\[ q_4=\dfrac{2(p_3^2+p_6^2)p_1}{\sqrt{-\lambda}},
\] 
with $\lambda=-4(p_1^2 \, (p_3^2+p_6^2)+(p_2\,p_6-p_3\,p_5)^2)$.
We suppose that $\psi_+$ is stable with $\lambda <0$.
Then, $q_4=0$ if and only if $p_1=0$.
Since $p_1$ has to be equal to zero, it can be shown that the compatibility condition $\psi_+ \wedge \omega=0$
is equivalent to the following system of equations:
\begin{equation}\label{CompDiag}
\begin{cases}
h_3p_3+h_4p_6=0, \\
h_3p_2+h_4p_5=0.
\end{cases}
\end{equation}

Moreover, the positive-definiteness of $g$ implies $h_1>0$.
Then, the normalization condition $\psi_+ \wedge\psi_-=\frac{2}{3}\omega^3$ is equivalent to
\begin{equation}\label{normSU3}
\lvert p_2p_6-p_3p_5\rvert=h_1(h_2^2+h_3^2+h_4^2).
\end{equation}
The balanced condition ${\text{d}}\omega^2=0$ is satisfied if and only if 
\begin{equation}\label{balancedSU3}
2\sqrt{3}h_1h_2+(h_2^2+h_3^2+h_4^2)'=0.
\end{equation}
Finally, the K\"ahler condition ${\text{d}}\omega=0$ holds if and only if 
\begin{equation}\label{KahlerSU3}
\begin{cases}
h_3=h_4=0\\
\sqrt{3}h_1+h_2'=0.
\end{cases}
\end{equation}
Multiplying \ref{normSU3} by $h_4$ and using \ref{CompDiag}, we obtain $h_4h_1(h_2^2+h_3^2+h_4^2)=0$ and, since $h_1>0$ and $h_2=h_3=h_4=0$ would imply $\omega^3=0$, we necessarily have $h_4=0$.
Then, \ref{CompDiag} implies 
\[
\begin{cases}
h_3p_3=0\\
h_3p_2=0,
\end{cases}
\]
from which it follows $h_3=0$ since $p_2=p_3=0$ would imply $\lambda=0$.
Then, \ref{balancedSU3} reads $h_2(\sqrt{3}h_1+h_2')=0$ and, since $h_2\neq 0 $, in order to have $\omega^3\neq 0$, we have $\sqrt{3}h_1+h_2'=0$, namely ${\text{d}}\omega=0$.
Therefore, any $G$-invariant balanced 
$\text{SU}(3)$-structure on the corresponding $M$ is necessarily K\"ahler.
This concludes case (c.3).

\subsection{Case (a.1)}\label{Case (a.1)}  $\mathfrak{g}=\mathfrak{su}(2)\oplus 2\mathbb{R}$, $\mathfrak{k}=\{0\}$. 

Since $\mathfrak{k}=\{0\}$, we can write $T_pM\cong \left<e_1|_p\right> \oplus \hat{\mathfrak{g}}\big|_p$, for each $p\in M^{\text{princ}}$.
Moreover, every $k$-form $\alpha$ on $M^{\text{princ}}$ of the form
\[ \alpha=\sum_{1\leq i_1 < \ldots < i_k\leq 6} \alpha_{i_1\ldots i_k} e^{i_1\ldots i_k},
\]
where $\alpha_{i_1\ldots i_k} \in C^{\infty}(\overset{\circ}{I})$, is $G$-invariant.
Let
\begin{equation}\label{psip,omg a.1}
	\omega \coloneqq \sum_{1 \leq i < j \leq 6} h_{ij}e^{ij}, \qquad \psi_+ \coloneqq \sum_{1\leq i < j < k \leq 6}p_{ijk} e^{ijk}
\end{equation}
be a pair of generic $G$-invariant forms on $M^{\text{princ}}$ of degree two and three, respectively, with coefficients $h_{ij}, p_{ijk} \in C^{\infty}(\overset{\circ}{I})$.
If we choose a  $\mathcal{B}$-orthogonal basis of $\mathfrak{su}(2)$ with vectors of constant norm, say 
\[
f_i=\left( \begin{array}{c|c|c}
	\tilde{e}_i & \vphantom{\begin{pmatrix} 0 \\ 0 \end{pmatrix}}  & \\ \hline
	\phantom{\begin{matrix} 0 & 0 \end{matrix}} & 0 & \\ \hline
	& & 0 \end{array}
\right)\hspace{-2pt}, \, i=1,2,3, 
\]
and extend it to a basis $\{f_i\}_{i=1,\ldots,5}$ of $\mathfrak{g}$, the structure equations with respect to the frame $\{e_i\}_{i=1,\ldots,6}$ of $M^{\text{princ}}$ are given by 
\[
{\text{d}}e^1=0, \quad {\text{d}}e^2=-2e^{34}, \quad {\text{d}}e^3=2e^{24}, \quad {\text{d}}e^4=-2e^{23}, \quad {\text{d}}e^5=0, 
\quad {\text{d}}e^6=0.
\]
Fix the volume form $\Omega\coloneqq-e^{1\ldots6}$.
We consider the forms given in \ref{psip,omg a.1} and set
\begin{align*}
	&p_{134}=p_{234}=1, \\
	&p_{136}=p_{235}=p_{246}=-p_{145}=e^{2t}, \\
	&h_{12}=\dfrac{3}{2}\dfrac{e^{4t}}{\sqrt{9+3e^{6t}}}, \\
	&h_{34}=-\dfrac{1}{3}\left(-3+\sqrt{9+3e^{6t}}\right)e^{-2t},\\
	&h_{35}=h_{36}=h_{46}=-h_{45}=1,\\
	&h_{56}=2e^{2t},
\end{align*} 
for each $t\in (-1,1)$,
and all other coefficients equal to zero. Then, by performing the change of variable
\[
\tilde{t}(t)\coloneqq \int_0^t a(s) {\text{d}}s, \quad a(s)=\sqrt{\frac{3}{2}}(9+3\,e^{6t})^{-\frac{1}4}e^{2t},
\]
one can easily check that the resulting pair $\left(\omega,\psi_+\right)$ defines a $G$-invariant balanced non-K\"ahler $\text{SU}(3)$-structure on the corresponding $M^{\text{princ}}$.
With respect to the $t$ parameter, the metric on $M^{\text{princ}}$ 
is represented by the matrix
\[
(g_{ij})=
\begin{pmatrix} 
	\dfrac{3}{2}\dfrac{e^{4t}}{\sqrt{9+3e^{6t}}} & 0 & 0 & 0 & 0 & 0 \\
	0 & \dfrac{3}{2}\dfrac{e^{4t}}{\sqrt{9+3e^{6t}}} & 0 & 0 & 0 & 0 \\
	0 & 0 &   \dfrac{3+ \sqrt{9+3e^{6t}}}{3 e^{2t}} & 0 & 1 & -1 \\
	0 & 0 & 0 & \dfrac{3+ \sqrt{9+3e^{6t}}}{3 e^{2t}} & 1 & 1 \\
	0 & 0 & 1 & 1 & 2e^{2t} & 0 \\
	0 & 0 & -1 & 1 & 0 & 2e^{2t}
\end{pmatrix}.
\]

However, using the results from \cite{Verdiani}, we can check that this example cannot be extended to the singular orbits to give a smooth metric on the whole manifold.
This concludes the proof of Theorem A.

\section{Proof of Theorem B}

We will finally prove our main theorem.
\begin{theoremB}
Let $M$ be a six-dimensional simply connected cohomogeneity one manifold under the almost effective action of a connected Lie group $G$, and let $K$ be the principal isotropy group. Assume $(\mathfrak{g}, \mathfrak{k}) \neq (\mathfrak{su}(2) \oplus \mathfrak{su}(2), \Delta \mathbb{R})$.
Then $M$ admits no $G$-invariant balanced non-K\"ahler $\text{\normalfont SU}(3)$-structures.
\end{theoremB}

By Theorem A, we only need to discuss if there exist balanced
non-K\"ahler $\text{\normalfont SU}(3)$-structures of cohomogeneity one arising as the compactification of the principal part determined by the pair $\left( \mathfrak{g}, \mathfrak{k} \right)=\left(\mathfrak{su}(2)\oplus 2\mathbb{R},\{0\} \right)$.
The question whether there is any such manifold with $(\mathfrak{g}, \mathfrak{k}) = (\mathfrak{su}(2) \oplus \mathfrak{su}(2), \Delta \mathbb{R})$ that admits a $G$-invariant balanced non-K\"ahler $\text{\normalfont SU}(3)$-structure remains open.

By \cite{Hoelscher2}, a six-dimensional compact simply connected cohomogeneity one manifold $M$ whose corresponding principal part is given by the pair $\left(\mathfrak{g},\mathfrak{k}\right)=\left(\mathfrak{su}(2)\oplus 2\mathbb{R},\{0\}\right)$ at the Lie algebra level is $G$-equivariantly diffeomorphic to the product of two three-dimensional spheres, i.e., $M\cong S^3\times S^3$.
If we denote by $H_i$ for $i=1,2$, the singular isotropy groups for the $G$-action on $M$ and by $\mathfrak{h}_i=\text{Lie}(H_i)$, for $i=1,2$, their Lie algebras, we have that both $\mathfrak{h}_1$ and $\mathfrak{h}_2$ are isomorphic to $\mathbb{R}$ so that both the singular orbits of $M$ are four-dimensional compact submanifolds of $M$.
By Michelsohn's obstruction \cite[Corollary 1.7]{Mic82}, if $M$ admitted any $4$-dimensional compact complex submanifold $S$, then $M$ would not admit a balanced metric. 
Therefore, we can make a few considerations by focusing on a tubular neighborhood of one singular orbit $G\supset H\supset K$ at a time. In particular, we divide the discussion depending on the immersion of $\mathfrak{h}\subset\mathfrak{g}$.
Let $S$ be the singular orbit given by the group diagram $G\supset H\supset K$. We notice that if $S$ is $J$-invariant, a complex structure on $M$ would give rise to a complex structure on $S$, so we can discard all these cases by Michelsohn's obstruction.
We have that $T_qM=T_qS \oplus V$ where $V=T_qS^{\perp}$ is the slice at $q\in S$; since $S$ is four-dimensional, $V$ is always a $2$-plane. We recall that the $H$-action on $T_qS$ is given by the adjoint representation while the $H$-action on $V$ is given via slice representation (which is determined from the embedding $K \subset H$), and since $V$ is two-dimensional, this action is just a rotation on $V$  of a certain weight, say $a$. 
Let us start by considering the case when $\mathfrak{h}$ is contained in the center of $\mathfrak{g}$, $\xi(\mathfrak{g})$.
In this case, the $H$-action on $T_qS$ is trivial.
Therefore, $T_qS$ and $V$ are inequivalent modules of the $H$-action on $T_qM$ and, since $J$ commutes with the $H$-action, $J$ preserves $T_qS$ for each $q\in S$, i.e., $S$ is an almost complex manifold and we may apply Michelsohn's obstruction to discard this case.
Therefore, we may suppose that $\mathfrak{h}$ has a non-trivial component in the $\mathfrak{su}(2)$-factor of $\mathfrak{g}$.
In particular, since $\text{rank}(\mathfrak{su}(2))=1$ and the adjoint action ignores components in the center of $\mathfrak{g}$, we will assume for our discussion, using the notation from Sect.\ \ref{Case (a.1)}, that $\mathfrak{h}=\left<f_1\right>$. 
Moreover, if we denote by $\mathfrak{m}$ the tangent space to $S$ via the usual identification, the decomposition of $\mathfrak{m}$ in irreducible $H$-modules is given by 
\[
\mathfrak{m}=l_0\oplus l_1,
\]
where $H$ acts on $l_0$ trivially and on $l_1$ via rotation of speed $d$.
The assumption $\mathfrak{h}=\left<f_1\right>$ does not change our discussion, since more generally if $h = \left< f_1 + X \right>$, where $X \in \mathbb{R} \oplus \mathbb{R} $, $ \{f_1 + X, f_2 , f_3 , f_4 , f_5 \}$ are again a basis for $\mathfrak{g}$ and $\mathfrak{m} \cong \left< f_2, f_3, f_4, f_5 \right>$ so the $H$-action on $\hat{\mathfrak{m}}|_q$ is again the adjoint representation $ad_\mathfrak{h}$ which splits $T_q S$ as sum of $l_0$ and $l_1$ as before. 
If the integer $a$ is different from $d$ the modules $l_0$, $l_1$ and $V$ are inequivalent for the $H$-action and again, since $J$ commutes with the $H$-action, it cannot exchange two different modules. In particular $J(T_qS)\subseteq T_qS$  and we may apply Michelsohn's obstruction as before.
For the remaining case $a=d$, we have that the two modules $l_1$ and $V$ are equivalent, hence, $J(l_1\oplus V)\subseteq l_1\oplus V$ but not necessarily $J(l_1)\subseteq l_1$. In particular, when this case occurs, the orbit $S$ is not $J$-invariant, and we do not have obstructions to the existence of balanced metrics.
Therefore, from now on, we assume this is the case. 

Let $\partial / \partial x$ be a vector field such that $\left( \xi|_q,  \partial / \partial x|_q \right)$ is an orthonormal basis for the slice $V$ and $T_q^*M = \langle e^1|_q, {\text{d}}x|_q, e^3|_q, e^4|_q, e^5|_q, e^6|_q \rangle$.
Let $\varphi : \mathfrak{h} \rightarrow \text{End}(T_qM)$ be the $\mathfrak{h}$-action on $T_qM$. Then in order to have $l_1$ and $V$ $\mathfrak{h}$-equivalent, $\varphi (f_1)^*$ acts on 1-forms given in the previous basis as 
\[
\varphi (f_1)^*=
\begin{pmatrix} 
	0 &{}\quad 1 &{}\quad 0 &{}\quad 0 &{}\quad 0 &{}\quad 0 \\
	-1 &{}\quad 0 &{}\quad 0 &{}\quad 0 &{}\quad 0 &{}\quad 0 \\
	0 &{}\quad 0 &{}\quad 0 &{}\quad 1 &{}\quad 0 &{}\quad 0 \\
	0 &{}\quad 0 &{}\quad -1 &{}\quad 0 &{}\quad 0 &{}\quad 0 \\
	0 &{}\quad 0 &{}\quad 0 &{}\quad 0 &{}\quad 0 &{}\quad 0 \\
	0 &{}\quad 0 &{}\quad 0 &{}\quad 0 &{}\quad 0 &{}\quad 0
\end{pmatrix}.
\]

Fix the volume form $\Omega = e^{1\ldots6}$ and consider 
the three form
$$
\begin{array}{ll}
	\psi_+ &:= p_1 e^{123} + p_2 e^{124} + p_3 e^{125} + p_4 e^{126} + p_5 e^{134} + p_6 e^{135} + p_7 e^{136} + p_8 e^{145}\\&{}\quad + p_9 e^{146} + p_{11} e^{234} + p_{12} e^{235} + p_{13} e^{236} + p_{14} e^{245} + p_{15} e^{246} +p_{16} e^{256}\\
	&{}\quad+ p_{17} e^{345} + p_{18} e^{346} + p_{19} e^{356} +p_{20} e^{456},
\end{array}
$$
where $p_j \in C^\infty((-1, 1))$ for any $j = 1, \dots , 20$.

The condition ${\text{d}} \psi_+=0$ is equivalent to the following ODE system:
\begin{equation}\label{eq:closure}
	\begin{cases}
		p_{11}'=0, \\
		p_{12}' +2 p_8=0, \\
		p_{13}' +2 p_9=0, \\
		p_{14}' -2 p_6=0, \\
		p_{15}' -2 p_7=0, \\
		p_{17}' +2 p_3=0, \\
		p_{18}' +2 p_4=0, \\
		p_{16}=p_{19}=p_{20}=0.
	\end{cases}
\end{equation}
From now on, we will assume $p_{16}=p_{19}=p_{20}=0$.

Let the slice be $V \cong \mathbb{R}^2$ so that the singular point $q \in \mathcal{O}_1$ is identified with $0 \in \mathbb{R}^2$, and let $r:V \rightarrow \mathbb{R}$ be the  radial distance, such that for $v=(v_1, v_2) \in V$, $r(v) =  | v|  = \sqrt{v_1^2 + v_2^2}$. Then $r \not\in C^\infty (V)$, and neither are the odd powers of $r$. Via the exponential map, we can identify $t+1$ with the radial distance $r$. 

Let $\alpha$ be a $G$-invariant 1-form on $M$. 
Then,
$$
\alpha(t)= \sum_{i=1}^6 \alpha_i(t) e^i,
$$
for $t \in (-1,1)$ and some smooth functions $\alpha_i$, $i=1, ..., 6$. This expression has to extend smoothly to $t=-1$. 
In particular, the Taylor expansion of $\alpha_k(t)$ around $t=-1$ for $k \geq 2 $ only has even powers of $t+1$:
$$
\alpha_k(t) \sim \sum_{n>1} a_{k, 2n} (t+1)^{2n}.
$$
Now for $2 \leq i < j<k \leq 6$ fixed,  the $e^{ijk}$-coefficients  extend smoothly to $t=-1$. Hence,
\[
p_{12}(t) \sim \sum_{n > 1} a_{2n} (t+1)^{2n},
\]
as well as for the Taylor expansion of $p_{13}(t), p_{14}(t)$ and $p_{15}(t)$ around $t=-1$.
Therefore,
$\lim_{t \rightarrow -1} p_{12}'(t)= \lim_{t \rightarrow -1} p_{13}'(t)= \lim_{t \rightarrow -1} p'_{14}(t)= \lim_{t \rightarrow -1} p'_{15}(t)=0$. 
From \ref{eq:closure}, we obtain that 
$\lim_{t \rightarrow -1} p_6(t)= \lim_{t \rightarrow -1} p_7(t)= \lim_{t \rightarrow -1} p_8(t)= \lim_{t \rightarrow -1} p_9(t)=0$. 

The three form $\psi_+$ at $t=0$ has to be $H$-invariant, and hence can be written as
$$
\begin{array}{ll}
\rho &=
c_3 e^1 \wedge {\text{d}}x \wedge e^5 +c_4 e^1 \wedge {\text{d}}x \wedge e^6 + c_6 e^{135} +c_7 e^{136} \\
&{}\quad+c_8 e^{145} + c_9 e^{146} - c_8 {\text{d}}x \wedge e^{35} -c_9 {\text{d}}x \wedge e^{36} \\
&{}\quad+c_6 {\text{d}}x \wedge e^{45} + c_7 {\text{d}}x \wedge e^{46} +c_{17} e^{345} +c_{18} e^{346},
\end{array}
$$
for some $c_3, c_4, c_6, c_7, c_8, c_9, c_{17}, c_{18} \in \mathbb{R}$. 
But $c_i = \lim_{t \rightarrow -1} p_i(t)=0$ for $i=6, 7, 8, 9$. Therefore, one can easily compute that
$$
\lambda |_{t=-1}= (c_{18} c_3 - c_{17} c_4)^2 
\geq0.
$$
This concludes case (a.1).

We note that it is possible to reach a contradiction by just studying the behavior around one of the singular orbits. However, if we do not use the information coming from Michelsohn's obstruction, the computations get significantly more complicated.
The main point is that from ${\text{d}} \psi_-=0$ and using the stability condition $\lambda <0$, we get $p_{10}=0$.
If we assume this too, the three form $\psi_+$ at $t=-1$ can be written as
$$
\begin{array}{ll}
	\rho &= c_1 e^{1} \wedge {\text{d}}x \wedge e^3 + c_2 e^{1} \wedge {\text{d}}x \wedge e^4 + c_3 e^{1} \wedge {\text{d}}x \wedge e^5 + c_4 e^{1} \wedge {\text{d}}x \wedge e^6 \\
	&{}\quad+ c_5 e^{134} + c_6 e^{135} + c_7 e^{136} + c_8 e^{145} + c_9 e^{146}  \\
	&{}\quad+ c_{11} {\text{d}}x \wedge e^{34} + c_{12} {\text{d}}x \wedge e^{35} + c_{13} {\text{d}}x \wedge e^{36} + c_{14} {\text{d}}x \wedge e^{45} + c_{15} {\text{d}}x \wedge e^{46} \\
	&{}\quad+ c_{17} e^{345} + c_{18} e^{346},
\end{array}
$$
for some $c_i \in \mathbb{R}$, $i=1,\ldots,18, i\neq 10,16$. 
Then, once again we find that $\lambda |_{t=-1}= (c_{18} c_3 - c_{17} c_4)^2 
\geq0$ which finishes the case.

\begin{remark} 
	We also note that in case (a.1) and when $\mathfrak{h}= \mathbb{R}$, we can remove the hypothesis of simply connectedness from the non-compact case and still get a non existence result.
	Let $M$ be a six-dimensional non-compact cohomogeneity one manifold under the almost effective action of a connected Lie group $G$ and let $K,H$ be the principal and singular isotropy groups, respectively, with $\left( \mathfrak{g}, \mathfrak{h}, \mathfrak{k} \right)=\left(\mathfrak{su}(2)\oplus 2\mathbb{R},\mathbb{R},\{0\} \right)$.
	Then, $M$ admits no $G$-invariant balanced non-K\"ahler $\text{\normalfont SU}(3)$-structures.
\end{remark}

From Theorem B, we get the following Corollary.
\begin{corollary}
There is no non-K\"ahler balanced $\normalfont\text{SU}(3)$-structure on $S^3 \times S^3$ which is invariant under a cohomogeneity one action.
\end{corollary}

\begin{remark}
    In the non-simply-connected case, and under the same assumptions of Theorem B, by Theorem A, we can discard cases (b.1) and (c.3), as these do not admit local solutions to conditions \normalfont{(1)--(7)}. Moreover, as observed in \cite[Section 3]{Pod}, one can also rule out cases (b.3) and (c.2), as the $G$-action would not be almost effective, as well as case (c.1) since it would give rise to a three-dimensional $J$-invariant subspace, a contradiction.
\end{remark}

Data sharing is not applicable to this article as no datasets were generated or analyzed during the current study.

\begin{footnotesize}

\end{footnotesize}

\begin{thebibliography}{00}


\bibitem{AleBettiol}
\textsc{Alexandrino, M.M., Bettiol, R.G.}: \emph{Lie Groups and Geometric Aspects of Isometric Actions}. Springer, Berlin (2015)

\bibitem{Alexandrov00}
\textsc{Alexandrov, B., Ivanov, S.}: \emph{Vanishing theorems on Hermitian manifolds}. Differ. Geom. Appl. \textbf{14}, 251--265 (2001)

\bibitem{BedVezzSu3}
\textsc{Bedulli, L., Vezzoni, L.}: \emph{The Ricci tensor of SU(3)-manifolds}. J. Geom. Phys. \textbf{57}, 1125--1146 (2007)

\bibitem{BedVezz}
\textsc{Bedulli, L., Vezzoni, L.}:  \emph{A parabolic flow of balanced metrics}. J. Reine Angew. Math. \textbf{723}, 79-99 (2017)

\bibitem{BerBer} 
\textsc{B\'{e}rard Bergery, L.}: \emph{Sur des nouvelles varietes Riemanniennes d'Einstein}. Publ. de Inst. E. Cartan \textbf{6} 1--60 (1982)

\bibitem{Fei}
\textsc{Fei, T.}: \emph{A construction of non-K\"ahler Calabi-Yau manifolds and new solutions to the Strominger system}. Adv. Math. \textbf{302} 529--550 (2016)

\bibitem{FeiHuangPicard}
\textsc{Fei, T., Huang, Z., Picard S.}: \emph{A construction of infinitely many solutions to the Strominger system}. J. Differ. Geom. \textbf{117} 23--39 (2021)

\bibitem{FeiYau}
\textsc{Fei, T., Yau, S.-T.}: \emph{Invariant solutions to the Strominger system on complex Lie groups and their quotients}. Commun. Math. Phys.  \textbf{338}  1183--1195 (2015)

\bibitem{SU2}
\textsc{Fern\'{a}ndez, M., Tomassini, A., Ugarte, L., Villacampa, R.}: \emph{Balanced Hermitian metrics from SU(2)-structures}. J. Math. Phys. \textbf{50}, 033507 (2009)

\bibitem{FinoGranVezz}
\textsc{Fino, A., Grantcharov, G.,  Vezzoni, L.}: \emph{Astheno-K\"ahler and balanced structures on fibrations}. Int. Math. Res. Not. \textbf{22}  7093--7117 (2019)

\bibitem{FinoVezz1}
\textsc{Fino, A., Vezzoni, L.}: \emph{Special Hermitian metrics on compact solvmanifolds}. J. Geom. Phys. \textbf{91} 40--53 (2015)

\bibitem{FinoVezz2}
\textsc{Fino, A., Vezzoni, L.}: \emph{On the existence of balanced and SKT metrics on nilmanifolds}. Proc.
Am. Math. Soc. \textbf{144} 2455--2459 (2016)

\bibitem{FH17}
\textsc{Foscolo, L., Haskins, M.}: \emph{New $G_2$-holonomy cones and exotic nearly K\"ahler structures on $S^6$ and $S^3 \times S^3$}. Ann. of Math. (2) \textbf{185} 59--130 (2017)

\bibitem{FuLiYau}
\textsc{Fu, J., Li, J., Yau,S.-T.}: \emph{Balanced metrics on non-K\"ahler Calabi-Yau threefolds}. J. Differ.
Geom. \textbf{90} 81--129 (2012)

\bibitem{FuTsengYau}
\textsc{Fu, J.-X., Tseng, L.-S., Yau,S.-T.}: \emph{Local heterotic torsional models}. Commun. Math. Phys. \textbf{289} 1151-1169(2009)

\bibitem{FuYau}
\textsc{Fu, J.-X., Yau, S.-T.}: \emph{The theory of superstring with flux on non-K\"ahler manifolds and the complex
Monge-Amp\'{e}re equation}. J. Differ. Geom. \textbf{78}  369--428 (2008) 

\bibitem{Mario}
\textsc{Garcia-Fernandez, M.}: \emph{T-dual solutions of the Hull-Strominger system on non-K\"ahler threefolds}. J. reine angew. Math. \textbf{766}  137--150 (2020)

\bibitem{Grantcharov08}
\textsc{Grantcharov, D., Grantcharov, G., Poon, Y.S.}: \emph{Calabi-Yau connections with torsion on toric bundles}. J. Differ. Geom. \textbf{78} 13--32 (2008)

\bibitem{Gran11}
\textsc{Grantcharov, G.}: \emph{Geometry of compact complex homogeneous spaces with vanishing first Chern class}. Adv. Math. \textbf{226} 3136-3159 (2011)

\bibitem{Hitchin}
\textsc{Hitchin, N.}: \emph{The geometry of three-forms in six dimensions}. J. Differ. Geom. \textbf{55} 547--576 (2000) 

\bibitem{Hoelscher}
\textsc{Hoelscher, C. A.}: \emph{Classification of cohomogeneity one manifolds in low dimensions}. Pac. J. Math. \textbf{246} 129--185 (2010)

\bibitem{Hoelscher2}
\textsc{Hoelscher, C. A.}:  \emph{Diffeomorphism type of six-dimensional cohomogeneity one manifolds}. Ann. Global Anal. Geom. \textbf{38} 1--9 (2010)

\bibitem{Hull}
\textsc{Hull, C.}: \emph{Superstring compactifications with torsion and space-time supersymmetry}. In Turin 1985 Proceedings ``Superunification and Extra Dimensions" 347--375 (1986)

\bibitem{Mic82}
\textsc{Michelsohn, M. L.}: \emph{On the existence of special metrics in complex geometry}. Acta Math. \textbf{149} 261-295
(1982)

\bibitem{OtalUgaVilla}
\textsc{Otal, A., Ugarte, L., Villacampa, R.}: \emph{Invariant solutions to the Strominger system and the heterotic equations of motion}. Nucl. Phys. B \textbf{920} 442--474 (2017)

\bibitem{PhongPicardZhang}
\textsc{Phong, D. H., Picard, S., Zhang, X.}: \emph{A flow of conformally balanced metrics with K\"ahler fixed points}. Math. Ann. \textbf{374}  2005--2040 (2019)

\bibitem{Pod}
\textsc{Podest\`{a}, F., Spiro, A.}: \emph{6-dimensional nearly K\"ahler manifolds of cohomogeneity one}. J. Geom. Phys. \textbf{60} 156--164 (2010)

\bibitem{Pujia}
\textsc{Pujia, M.}: \emph{The Hull-Strominger system and the Anomaly flow on a class of solvmanifolds}. J. Geom. Phys. \textbf{170}  104352 (2021)

\bibitem{PujiaUgarte}
\textsc{Pujia, M., Ugarte, L.}: \emph{The Anomaly flow on nilmanifolds}. 	arXiv:2004.06744.

\bibitem{Reichel}
\textsc{Reichel, W.}: \emph{\"Uber die Trilinearen Alternierenden Formen in 6 und 7 Ver\"ander-lichen}, Ph.D. thesis, Greifswald (1907).
 
\bibitem{Strominger} 
\textsc{Strominger,A.}: \emph{Superstrings with torsion}. Nucl. Phys. B \textbf{274} 253--284 (1986)

\bibitem{UgaVilla}
\textsc{Ugarte, L., Villacampa, R.}: \emph{Balanced Hermitian geometry on 6-dimensional nilmanifolds}. Forum Math. \textbf{27}  1025--1070 (2015)

\bibitem{Verdiani}
\textsc{Verdiani, L., Ziller, W.,}: \emph{Smoothness Conditions in Cohomogeneity one manifolds}, Transform. Groups \textbf{27}  311--342 (2022)
 
\bibitem{Wang}
\textsc{Wang, H.-C.}: \emph{Complex parallisable manifolds}. Proc. Amer. Math. Soc. \textbf{5}  771--776 (1954)

\bibitem{Ziller}
\textsc{Ziller, W.}: \emph{On the geometry of cohomogeneity one manifolds with positive curvature}. In \emph{Riemannian Topology and Geometric Structures on Manifolds}, Progress in Mathematics \textbf{271}, Birkh\"auser (2009)


\end{thebibliography}
\end{document}